\documentclass[11pt]{article}
\usepackage{amsmath, amsfonts, amsthm, amssymb, multicol}
\usepackage{graphicx}
\usepackage{float}
\usepackage{verbatim}

\allowdisplaybreaks

\usepackage{xcolor}
\usepackage[pagebackref]{hyperref}
\hypersetup{
   colorlinks,
    linkcolor={red!60!black},
    citecolor={blue!60!black},
    urlcolor={blue!90!black}
}

\numberwithin{equation}{section}

\newtheorem{thm}{Theorem}[section]
\newtheorem{lma}[thm]{Lemma}
\newtheorem{cor}[thm]{Corollary}

\renewcommand{\epsilon}{\varepsilon}

\newcommand{\e}{\varepsilon}

\renewcommand{\le}{\leqslant}
\renewcommand{\geq}{\geqslant}
\renewcommand{\leq}{\leqslant}

\newcommand{\ubd}{\overline{\dim}_{\textup{B}}}
\newcommand{\lbd}{\underline{\dim}_{\textup{B}}}

\newcommand{\ad}{\dim_{\mathrm{A}} }
\newcommand{\qad}{\dim_{\mathrm{qA}} }
\newcommand{\as}{\dim^\theta_{\mathrm{A}} }

\newcommand{\hd}{\dim_{\mathrm{H}}  }

\newcommand{\pd}{\dim_{\mathrm{P}}}

\newcommand{\rn}{\mathbb{R}^n}

\newcommand{\be}{\begin{equation}}
\newcommand{\ee}{\end{equation}}

\title{ \vspace{-20mm} Assouad dimension influences the box and packing dimensions of orthogonal projections}

\author{Kenneth J. Falconer, Jonathan M. Fraser and Pablo Shmerkin}

\begin{document}


\maketitle

\begin{abstract}
\noindent We present several applications of the Assouad dimension, and the related quasi-Assouad dimension and Assouad spectrum, to the box and packing dimensions of orthogonal projections of sets.  For example, we show that if the (quasi-)Assouad dimension of  $F \subseteq \rn$ is no greater than $m$, then the box and packing dimensions of $F$ are preserved under orthogonal projections onto almost all $m$-dimensional subspaces. We also show that the threshold $m$ for the (quasi-)Assouad dimension is sharp, and  bound the dimension of the exceptional set  of projections strictly away from the dimension of the Grassmannian.
\end{abstract}

\section{Introduction and results}
The relationship between the dimension of a Borel set $F \subseteq \rn$ and its projections onto $m$-dimensional subspaces goes back to Marstrand \cite{mar} and Mattila \cite{mat}  who showed that
\[
\hd \pi_V F = \min\{m, \hd  F\}
\]
for almost all  $V\in G(n,m)$ with respect to the natural invariant measure on the Grassmannian $G(n,m)$, where $\pi_V: \rn\to V$ is orthogonal projection onto $V$ and $\hd$ is Hausdorff dimension.

Finding the box and packing dimensions of projections of sets is more awkward. For a non-empty bounded  $F\subseteq \rn$,  let $N_r(F)$ be the minimum number of sets of diameter $r$ that can cover $F$. The {\em lower} and {\em upper box-counting dimensions} or {\em box dimensions} of $F$ are defined by
\[
\lbd F\ =\ \varliminf_{r\to 0} \frac{\log  N_r(F)}{-\log r}
\quad \mbox{ and }\quad  \ubd F\ = \ \varlimsup_{r\to 0} \frac{\log  N_r(F)}{-\log r}.
\]
(Note that this gives the same values for the dimensions as taking  $N_r(F)$ to be the least number of sets of diameter at most $r$ that can cover $F$ or other equivalent definitions, see \cite{falconer}). The {\em packing dimension} of a (not-necessarily bounded) set may be defined in terms of upper box dimension as
\be\label{packdim}
\pd F\ =\ \inf\big\{\sup_j  \ubd F_j : F\subseteq \cup_{j=1}^\infty F_j \text{ with }F_j \text{ compact}\big\},
\ee
see \cite{falconer}.
Although the values of $\lbd \pi_V F,  \ubd \pi_V F$ and $\pd \pi_V F$ are constant for almost all $V\in G(n,m)$,  this constant can take any value in the range
 \begin{equation} \label{boxdimsbounds}
\frac{\ubd F}{1+(1/m-1/n)\ubd F} \  \leq \  \ubd \pi_V F \leq \min \{m, \ubd F\},
\end{equation}
with analogous inequalities for lower box and packing dimensions.
These inequalities were established in \cite{falconerhowroyd, falconerhowroyd2, profiles} using dimension profiles, with examples showing them to be best possible in \cite{falconerhowroyd, jarvenpaa}, and recently a simpler approach using capacities was introduced \cite{falconerprofile}. For background on the dimensions of projections, see \cite{survey1,survey2}, and for general dimension theory, see \cite{falconer}.

In light of the fact that, in general, box and packing dimensions may drop below the upper bound in \eqref{boxdimsbounds} under almost all projections, it may be of interest to find geometric conditions that ensure that such a drop does not occur. In this paper we obtain a result of this type: we show that if the Assouad, or even the quasi-Assouad, dimension of $F \subseteq \rn$ is no greater than $m$, then the box and packing dimensions of $F$ are preserved under orthogonal projection onto almost all $m$-dimensional subspaces, i.e. there is equality on the right-hand side of \eqref{boxdimsbounds}. We also obtain estimates when  $\dim_\text{A} F >m$ as well as bounds on the dimension of the exceptional set of subspaces $V$.

 The \emph{Assouad dimension} of a non-empty  set $F \subseteq \rn$ is defined by
\begin{align*}
\dim_\text{A} F \ = \  \inf \bigg\{ \  \alpha &: \text{     there exists a constant $C >0$ such that,} \\
&  \text{for all $0<r<R $ and $x \in F$, }   \text{$ N_r\big( B(x,R) \cap F \big) \ \leq \ C \bigg(\frac{R}{r}\bigg)^\alpha$ } \bigg\}.
\end{align*}
The related   \emph{upper Assouad spectrum} is defined by
\begin{align*}
\overline{\dim}_\textup{A}^\theta F  \ = \    \inf \bigg\{ \  \alpha &: \text{   there exists  a constant  $C >0$ such that,} \\
 &  \text{for all $0<r \leq R^{1/\theta} < R<1$ and $x \in F$, }
  \text{$ N_{r} \big( B(x,R) \cap F \big) \ \leq \ C \bigg(\frac{R}{r}\bigg)^\alpha$ } \bigg\}
\end{align*}
where $\theta \in (0,1)$. If one replaces $r \leq R^{1/\theta}$ with $r = R^{1/\theta}$, then one obtains the Assouad spectrum $\as F$, see \cite{Spectraa}, but it was proved in \cite{canadian} that
\[
\overline{\dim}_\textup{A}^\theta F = \sup_{\theta' \in (0,\theta)} \dim_\textup{A}^{\theta'}  F
\]
and so we are able to rely on the theory of $\as F$, which is somewhat more developed.  The upper Assouad spectrum is clearly non-decreasing in $\theta$ but the Assouad spectrum need not be.  However, in most commonly studied situations it is non-decreasing and therefore the two spectra coincide.  Finally, the \emph{quasi-Assouad dimension} is defined by
\[
\qad F = \lim_{\theta \nearrow 1} \overline{\dim}_\textup{A}^\theta F.
\]
Generally,  for $\theta \in (0,1)$,
\[
\pd \! F  \leq \ubd F \leq \overline{\dim}_\textup{A}^\theta F  \leq \qad F \leq \ad F.
\]

With these definitions we may state our two basic theorems which will be proved in the next section using dimension profiles.

\begin{thm}\label{boxapp0}
Let $1\leq m < n$ and  $\theta \in (0,1)$. If  $F\subseteq \rn$ is bounded then, for almost all $V \in G(n,m)$,
\be\label{ubtheta}
\ubd \pi_V F \geq \ubd  F - \max \{0, \ \overline{\dim}_\textup{A}^\theta F-m, \ (\ad F-m)(1-\theta)\},
\ee
and the same conclusion holds with $\ubd$ replaced by $\lbd$. If $F$ is Borel the conclusion holds with $\ubd$ replaced by $\pd$.
\end{thm}

The following statement bounds the Hausdorff dimension of the exceptional set of projections in Theorem \ref{boxapp0} strictly away from $m(n-m)$, the dimension of $G(n,m)$.

\begin{thm}\label{boxapp5}
Let $1\leq m < n$, $s\in (0,m)$ and $\theta \in (0,1)$. If $F\subseteq \rn$ is bounded then
\begin{align}
\hd \Big\{ V \in G(n,m) : \ubd \pi_V F < \ubd F - \max \{0 & , \ \overline{\dim}_\textup{A}^\theta F-s, \ (\ad F-s)(1-\theta)\}\Big\} \nonumber\\
&\leq m(n-m)-(m-s),\label{exbound}
\end{align}
and the same conclusion holds with $\ubd$ replaced by $\lbd$. If $F$ is Borel  the conclusion holds with $\ubd$ replaced by $\pd$.
\end{thm}

The following corollaries follow easily from the theorems by choosing appropriate parameters.

\begin{cor}\label{boxapp2}
Suppose that $\qad F \leq \max\{m,\ubd F\}$. If $F\subseteq \rn$ is bounded then,
for   almost all $V\in G(n,m)$,
\be\label{equa}
\ubd \pi_V F =\min\{m,\ubd F\},
\ee
and more generally
\be\label{drop}
\ubd \pi_V F \geq \ubd  F - \max \{0, \ \qad F-m\}.
\ee
The same conclusion holds with $\ubd$ replaced by $\lbd$. If $F$ is Borel the conclusion holds with $\ubd$ replaced by $\pd$.
\end{cor}

\begin{proof}
The almost sure estimates \eqref{equa} and \eqref{drop} follow on letting $\theta \nearrow 1$ in \eqref{ubtheta}.  \end{proof}

The equality \eqref{equa} in Corollary \ref{boxapp2} can also be obtained using \cite[Proposition 4.5]{orp}.

\begin{cor}\label{boxapp4}
Suppose that $\qad F < m$. If $F\subseteq \rn$ is bounded  then
$$
\hd \{V \in G(n,m) : \ubd \pi_V F < \ubd  F\} \leq m(n-m)-(m-\qad F),
$$
and the same conclusion holds with $\ubd$ replaced by $\lbd$. If $F$ is Borel the conclusion holds with $\ubd$ replaced by $\pd$.
\end{cor}

\begin{proof}
Take $s= \qad F$ and let $\theta \nearrow 1$ in \eqref{exbound}.
\end{proof}

In the absence of a precise result, a natural question is when \eqref{drop} improves on the general lower bounds from \eqref{boxdimsbounds}.  A careful analysis of the lower bound yields many such situations.  We provide one instance, based on a knowledge of the Assouad dimension.  We exclude the range $\ad F  < \max\{m, \ubd F \}$ since this is covered by Corollary \ref{boxapp2}.

\begin{cor}\label{boxapp3}
Let $1\leq m<n$ be integers. If $F\subseteq \rn$ is bounded with $\ubd F \leq m$ and
\[
\max\{m, \ubd F \} \leq \ad F < \frac{(mn + 2\ubd F (n-m))m}{mn + \ubd F (n-m)},
\]
then, for   almost all $V\in G(n,m)$,
\[
\ubd \pi_V F \geq  \ubd  F -  \ \frac{(\ad F-m)\ubd F}{m} > \frac{\ubd F}{1+(1/m-1/n)\ubd F}.
\]
This result remains valid with $\ubd$ replaced by $\lbd$ throughout.   If $F$ is Borel the result remains valid with $\ubd$ replaced by $\pd$.
\end{cor}

\begin{proof}
This follows  from Theorem \ref{boxapp0}.  If $\ubd F = m$, then the bound follows by letting $\theta \searrow 0$ in \eqref{ubtheta}.  If $\ubd F < m$, then choose $\theta$ as large as possible such that $\overline{\dim}_\textup{A}^\theta F \leq m$.  By \cite[Proposition 3.1]{Spectraa} we know that
\[
\overline{\dim}_\textup{A}^\theta F \leq \frac{\ubd F}{1-\theta}
\]
and therefore we can always choose
\[
\theta = 1-\frac{\ubd F}{m}.
\]
Therefore by \eqref{ubtheta},   for   almost all $V\in G(n,m)$,
\begin{eqnarray*}
\ubd \pi_V F \geq \ubd  F - (\ad F-m)(1-\theta) &= &  \ubd  F -  \ \frac{(\ad F-m)\ubd F}{m} \\ \\
 &>& \frac{\ubd F}{1+(1/m-1/n)\ubd F},
\end{eqnarray*}
as required. The final strict inequality uses the assumption on the Assouad dimension.
\end{proof}

Figure 1 indicates our bounds for the almost sure box dimensions of projections depending on the pair $(\ubd F ,  \ad F)$ in different cases.

\begin{figure}[H]
	\centering
	\includegraphics[width=0.45\textwidth]{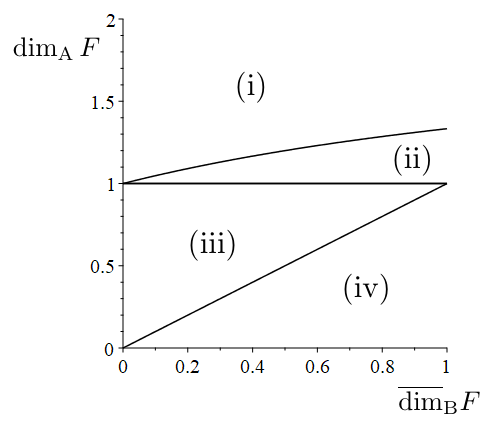}
\caption{Taking $m=1$ and $n=2$,  if the pair $(\ubd F ,  \ad F)$ lies in  region: (i) we get no information, (ii) Corollary \ref{boxapp3}  gives improvements on the general bounds \eqref{boxdimsbounds}, (iii)  Corollary \ref{boxapp2} gives $\min\{m,\ubd F\}$, (iv) is not possible, since $\ad F \geq \ubd F$.  The curve bounding regions (i) and (ii) is given by $y =  (2x+2)/(x+2)$. }
\end{figure}

The proof of Corollary \ref{boxapp3} involved choosing a particular $\theta$ in the `worst case scenario'.  If $\as F$ is known, then there may be a much better choice leading to better estimates in the  particular setting.  For example, better choices of $\theta$ always  exist if $F$ is a Bedford-McMullen carpet, see \cite{Spectrab}.

We remark that the examples in \cite[Lemma 19]{falconerhowroyd2} of sets $F\subseteq \mathbb{R}^n$ for which there is almost sure equality on the left-hand inequality of \eqref{boxdimsbounds}, all have Assouad dimension $\ad F = n$, consistent with our estimates.

\section{Dimension profiles and proofs of theorems}

We first review the relationship between the dimension profiles of  a set and the box dimensions of its projections. Then estimating the dimension profiles in terms of  (quasi-)Assouad dimensions will lead to the theorems in Section 1.

Dimension profiles may be defined in terms of capacities with respect to certain kernels \cite{falconerprofile}. For $s \in [0,n]$ and $r>0$ we define the  kernel
\be
 \phi_r^s(x)= \min\bigg\{ 1, \bigg(\frac{r}{|x|}\bigg)^s\bigg\} \quad (x\in \rn).\label{ker}
\ee
For a  non-empty compact  $F \subseteq \rn$,  the \emph{capacity}, $C_r^s(F)$, of $F$ with respect to this kernel is  given by
\[
\frac{1}{C_r^s(F)}\  = \ \inf_{\mu \in {\mathcal M}(F)}\int\int \phi_r^s(x-y)d\mu(x)d\mu(y)
\]
where  ${\mathcal M}(F)$ denotes the collection of Borel probability measures supported by $F$.  The double integral inside the infimum is called the {\it energy} of $\mu$ with respect to the kernel. The capacity of a general bounded set is taken to be that of its closure.
For bounded $F \subseteq \rn$ and $s>0$ we define the {\it lower} and {\it upper box dimension   profiles} of $F$ by
\[
\underline{\mbox{\rm dim}}_{\rm B}^s F \ =\   \varliminf_{r\to 0} \frac{\log  C_r^s(F)}{-\log r},\quad \overline{\mbox{\rm dim}}_{\rm B}^s F \ =\   \varlimsup_{r\to 0} \frac{\log  C_r^s(F)}{-\log r},
\]
and, analogously to the packing dimension \eqref{packdim},  the  {\it packing dimension profile} (for $F$ not necessarily bounded) by
\be\label{packprof}
\pd^s F\ =\ \inf\big\{\sup_j  \ubd^s F_j : F\subseteq \cup_{j=1}^\infty F_j \text{ with }F_j \text{ compact}\big\}.
\ee
In particular, by \cite[Corollary 2.5]{falconerprofile} if $s\geq n$ then
\[
\underline{\mbox{\rm dim}}_{\rm B}^s F \ =\  \lbd F, \ \  \overline{\mbox{\rm dim}}_{\rm B}^s F \ =\  \ubd F,
\ \  \pd^s \! F \ =\  \pd\!  F
\]
but for  $s < n$ the dimension profiles give the almost sure  dimensions of projections of sets as well as information on the size of the set of exceptional projections, as follows.

\begin{thm}{\rm \cite[Theorems 1.1, 1.2]{falconerprofile}}\label{mainA}

\noindent $(i)$ Let $1\leq m <n$ be an integer.  For almost all $V \in G(n,m)$, if $F\subseteq \rn$ is bounded
\[
\lbd \pi_V F = \lbd^m F, \ \ \text{and} \ \ \ubd \pi_V F = \ubd^m F,
\]
and  if $F\subseteq \rn$ is Borel
\[
  \pd\!  \pi_V F = \pd^m\!  F.
\]
\noindent $(ii)$ For $0<s<m$, if $F\subseteq \rn$ is bounded
\[
\hd \{ V \in G(n,m) : \lbd \pi_V F < \lbd^s F\} \leq m(n-m)-(m-s),
\]
\[
\hd \{ V \in G(n,m) : \ubd \pi_V F < \ubd^s F\} \leq m(n-m)-(m-s),
\]
and  if $F\subseteq \rn$ is Borel
\[
\hd \{ V \in G(n,m) : \pd\!  \pi_V F < \pd^s \! F\} \leq m(n-m)-(m-s).
\]
\end{thm}

The following theorem relates the dimension profiles to (quasi-)Assouad dimension. Combined with Theorem \ref{mainA} this will give the bounds stated in Section 1 for the box and packing dimensions of projections.

\begin{thm}\label{boxapp}
Let    $s\in (0,n]$  and $\theta \in (0,1)$. If $F\subseteq \rn$ is bounded then
\[
\lbd^s F \geq \lbd  F - \max \{0, \ \overline{\dim}_\textup{A}^\theta F-s, \ (\ad F-s)(1-\theta)\}
\]
and
\be\label{ineqprof}
\ubd^s F \geq \ubd  F - \max \{0, \ \overline{\dim}_\textup{A}^\theta F-s, \ (\ad F-s)(1-\theta)\}.
\ee
If $F\subseteq \rn$ is Borel then
\be\label{ineqpd}
\pd^s \! F \geq \pd  \! F - \max \{0, \ \overline{\dim}_\textup{A}^\theta F-s, \ (\ad F-s)(1-\theta)\}.
\ee

\end{thm}

\begin{proof}
We first consider upper box dimensions. We may assume for convenience that $|F|<1/2$, where $|F|$ denotes the diameter of $F$.  Throughout this proof we write $N_r(E)$ to denote the maximal size of an $r$-separated subset of $E$. Fix   $\alpha> \overline{\dim}_\textup{A}^\theta F$, $\beta> \ad F$ and let  $C>0$ be a constant such that for all $0<r<R<1$ and $x \in F$
\[
N_r(B(x,R) \cap F) \leq C \left(\frac{R}{r}\right)^\beta,
\]
and  for all $0<r \leq R^{1/\theta}<R<1$ and $x \in F$
\[
N_r(B(x,R) \cap F) \leq C \left(\frac{R}{r}\right)^\alpha.
\]
Let  $0< r<1$ and  $\{x_i\}_{i=1}^{N_r(F)}$ be a maximal $r$-separated set of points in $F$.  Place a point mass of weight $1/N_r(F)$ at  each $x_i$ and let the measure $\mu$ be the aggregate of these point masses so that $\mu(F) = 1$.

Write $D= \lceil \log_2(2|F|r^{-1})\rceil$ and $B=\lceil (1-\theta)\log_2(r^{-1})\rceil$ noting that for sufficiently small $r$, $1\leq B<D$.      For each $i$ the potential of $\mu$ at $x_i$ is
\begin{eqnarray*}
\int  \phi_r^s(x_i-y)d\mu(y)  &\leq& \sum_{k=0}^D 2^{-(k-1)s} \mu(B(x_i,2^k r)) \\ \\
&\leq& \sum_{k=0}^D 2^{-(k-1)s}\frac{1}{N_r(F)} N_r\big(B(x_i,2^k r)\cap F\big) \\ \\
&\leq&  \frac{2^s}{N_r(F)} \bigg( \sum_{k=B}^D 2^{-ks}C\Big(\frac{2^k r}{r}\Big)^\alpha \ + \ \sum_{k=0}^{B-1} 2^{-ks}C\Big(\frac{2^k r}{r}\Big)^\beta\bigg) \\ \\
&\leq& c\, \frac{\max\{1, \ r^{-(\alpha-s)}, \ r^{-(\beta -s)(1-\theta)}\}}{N_r(F) }
 \end{eqnarray*}
for a constant $c$ which is independent of $r$. Summing over the $x_i$, the energy of $\mu$ is
$$\int \int \phi_r^s(x-y)d\mu(x)d\mu(y) \leq c\frac{\max\{1, \ r^{-(\alpha-s)}, \ r^{-(\beta-s)(1-\theta)}\}}{N_r(F) }$$
and so the capacity $C_r^s(F)$ satisfies
\[
C_r^s(F)\geq c^{-1}N_r(F) \min\{1, \ r^{(\alpha-s)}, \ r^{(\beta-s)(1-\theta)}\} .
\]
Thus
\begin{eqnarray*}
 \overline{\mbox{\rm dim}}_{\rm B}^s F  &=&   \varlimsup_{r\to 0} \frac{\log  C_r^s(F)}{-\log r}\\ \\
& \geq& \varlimsup_{r\to 0} \frac{\log \left(c^{-1}N_r(F) \min\{1, \ r^{(\alpha-s)}, \ r^{(\beta-s)(1-\theta)}\} \right)}{-\log r}\\ \\
& =& \ubd  F - \max \{0, \ \alpha-s, \ (\beta-s)(1-\theta)\}.
\end{eqnarray*}
The conclusion for upper dimensions follows on taking $\alpha$ and $\beta$ arbitrarily close to $\overline{\dim}_\textup{A}^\theta F$ and $\ad F$ respectively.  For the lower box dimension case we take lower limits in the final inequalities.

Finally, we extend the conclusions for upper dimensions to packing dimensions.  Given a Borel set $F$ with $\pd F >\gamma$ there exists a compact $E \subseteq F$ such that $\pd (E\cap U)  >\gamma$ for every open set $U$ that intersects $E$, see for example \cite[Lemma 2.8.1]{bisper}. Let $\{U_i\}$ be a countable basis of open sets that intersect $E$. From \eqref{ineqprof},
\begin{eqnarray*}
\ubd^s (E\cap \overline{U_i})& \geq & \ubd  (E\cap \overline{U_i}) - \max \big\{0, \ \overline{\dim}_\textup{A}^\theta (E\cap \overline{U_i})-s, \ (\ad (E\cap \overline{U_i}) -s)(1-\theta)\big\}\\
& \geq &\gamma - \max \big\{0, \ \overline{\dim}_\textup{A}^\theta F -s, \ (\ad F -s)(1-\theta)\big\}
\end{eqnarray*}
for all $i$, using the monotonicity of $\overline{\dim}_\textup{A}^\theta$ and $\ad$.

Let $\{E_j\}$ be any countable cover of $E$ by compact sets. By Baire's category theorem, for some $j$,  $E\cap E_j$ contains a set that is open relative to $E$, so  $E\cap  \overline{U_i} \subseteq E\cap E_j$ for some $i$.
It follows from the definition of the packing dimension profile $\pd^s$ \eqref{packprof} that
$$\pd^s \! F \ \geq \ \pd^s \! E \ \geq \ \gamma - \max \big\{0, \ \overline{\dim}_\textup{A}^\theta F -s, \ (\ad F -s)(1-\theta)\big\}.$$
Taking $\gamma$ arbitrarily close to $\pd F$ gives \eqref{ineqpd}.
\end{proof}

\noindent {\it Proof of Theorems \ref{boxapp0} and \ref{boxapp5}.}
Theorem \ref{boxapp0} is immediate on substituting the inequalities of  Theorem \ref{boxapp} with $s=m$ in Theorem \ref{mainA}(i). Similarly Theorem \ref{boxapp5} follows using the estimates of Theorem \ref{boxapp} in
Theorem \ref{mainA}(ii).
\hfill $\Box$

\section{Sharpness of the threshold for the (quasi-)Assouad dimension}

To conclude the paper, we show that Corollary \ref{boxapp2} is sharp in the following sense:
\begin{lma} \label{example}
For all $s\in (m,n]$ and $t\in (0,s)$ there exists a compact set $F\subset\rn$ such that $\pd F=\ubd F =t$, $\ad F =s$ (in particular, $\qad F \le s$), and 
\[
\pd \pi_V F  \le \ubd \pi_V F \le \frac{mst}{m(s-t)+st} < \min(t,m)
\] 
for every $V\in G(n,m)$.
\end{lma}
This lemma says that, in order to guarantee that packing or upper box dimensions are preserved under almost all orthogonal projections (or indeed under even one orthogonal projection), it is not enough to assume that $\qad F \le s$, or even that $\ad F \le s$, if $s>m$ (while $s=m$ is enough by Corollary \ref{boxapp2}). However, how much the packing or box dimension of a typical projection $\pi_V F$, $V\in G(n,m)$, can drop from $\ubd F $ in terms of $\ubd F $ and $\ad  F $ remains an open problem, since the upper and lower bounds provided by Corollary \ref{boxapp2} and Lemma \ref{example} respectively, are in general quite far apart from each other. We note that when $s=n$, the upper bound given by Lemma \ref{example} agrees with the lower bound in \eqref{boxdimsbounds}, and is therefore sharp in this case. 

The construction of the set $F$ in the above lemma is based on sets defined by restricting the digits in dyadic expansions. Given a set $S\subseteq\mathbb{N}$, let
\[
X_S = \left\{ \sum_{k=1}^\infty a_k 2^{-k}: a_k \in\{0,1\} \text{ and } a_k=0  \text{ for all } k\in \mathbb{N}\setminus S \right\} \subset [0,1].
\]
We write  $\#F$ to denote the cardinality of a set $F$.  Recall the definition of (upper) density and Banach density of a subset  of $\mathbb{N}$:
\begin{align*}
\overline{d}(S) &= \limsup_{k\to\infty} \frac{\#(S\cap\{1,\ldots,k\})}{k},\\
\overline{d}_B(S) &= \limsup_{k\to\infty} \sup_{\ell\in\mathbb{N}} \frac{\#(S\cap\{\ell,\ell+1,\ldots,\ell+k-1\})}{k}.
\end{align*}

\begin{lma} \label{dyadic}
Given $S\subset\mathbb{N}$ and $n\in\mathbb{N}$,
\[
\pd X_S^n =\ubd X_S^n =\overline{d}(S)n,\quad \ad X_S^n  =\overline{d}_B(S)n.
\]
where $X_S^n \subseteq [0,1]^n$ is the $n$-fold product of $X_S$.
\end{lma} 
\begin{proof}
The claim for upper box dimension is almost immediate from the definition, see \cite[Example 1.4.2]{bisper} for details in the case $n=1$. The claim for packing dimension follows from the one on upper box dimension and \cite[Lemma 2.8.1]{bisper}. Finally, for the Assouad dimension formula, we note that if
\[
2^{-\ell-k}\le r < 2^{1-\ell-k} \le 2^{-\ell-1} < R \le 2^{-\ell}
\] 
then, for any $x\in X_S$, the set $X_S\cap B(x,R)$ can be covered by $C_n\cdot 2^{-n \#(S\cap \{\ell,\ldots,\ell+k-1\})}$ balls of radius $r$ and cannot be covered by fewer than a constant (depending on $n$) multiple of this number, so that $\ad(X_S)= \overline{d}_B(S)n$.
\end{proof}

\begin{proof}[Proof of Lemma \ref{example}]
Let $A\subseteq\mathbb{N}$ be a set with $\overline{d}(A)=\overline{d}_B(A)=s/n$; this is easily arranged. Let $(k_j)_{j\in\mathbb{N}}$ be a sequence of natural numbers satisfying
\be \label{rapidly-increasing}
(k_1 \cdots k_{j-1})/k_j\to 0 \quad\text{as}\quad j\to\infty.
\ee
Finally, set
\[
S = \bigcup_{j=1}^\infty (A+k_j) \cap \left\{ k_j,\ldots, \left\lfloor\frac{s}{s-t} k_j\right\rfloor \right\},
\]
and $F=X_S^n$. Here $A+k_j=\{ a+k_j:a\in A\}$. Using \eqref{rapidly-increasing},  we see that $\overline{d}(S)$ is realized along the sequence $\lfloor\tfrac{s k_j}{s-t}\rfloor$, and a calculation shows that $\overline{d}(S)=t/n$. Also, $\overline{d}_B(S)=\overline{d}_B(A)=s/n$ and hence, thanks to Lemma \ref{dyadic},
\[
\pd F =\ubd F =t, \quad \ad F =s.
\]
Now fix $V\in G(n,m)$, $\e>0$ and $k\in\mathbb{N}$. Pick $j$ such that $k_j \le k < k_{j+1}$. Provided $k$ (and therefore $j$) is large enough in terms of $\e, n, s$ and $t$, the set $F$ can be covered by
\[
2^{(n\prod_{i=1}^{j-1} \frac{t k_i}{s-t})}  \le 2^{\e k_j}
\]
cubes of side-length $2^{-k_j}$, where we used \eqref{rapidly-increasing}. Hence, $\pi_V F $ can be covered by $C_{n,m} 2^{m(k-k_j)} 2^{\e k_j}$ cubes of side-length $2^{-k}$. On the other hand, if $k > s k_j/(s-t)$, then (again assuming $k$ is large enough), $F$ can be covered by $2^{(t+\e) s k_j/(s-t)}$ cubes of side-length $2^{-k}$, and hence $\pi_V F $ can be covered by at most a constant $C_{n,m}$ multiple of that number. Up to the terms involving $\e$, the first bound is more efficient when
\[
k < \big(1+\tfrac{st}{m(s-t)}\big)k_j,
\]
otherwise the second bound is more efficient. Note that $1+\tfrac{st}{m(s-t)}> \tfrac{s}{s-t}$, since $s>m$. A short calculation shows that, in any case, $\pi_V F $ can be covered by
\[
C_{n,m} 2^{ (d+ \e C_{m,s,t})k  }, \quad d=\frac{mst}{m(s-t)+st},
\] 
cubes of side-length $2^{-k}$. Since $\e>0$ was arbitrary, this concludes the proof.
\end{proof}

\subsection*{Acknowledgements}

KJF and JMF were   supported by an \emph{EPSRC Standard Grant} (EP/R015104/1). JMF was supported by a \emph{Leverhulme Trust Research Project Grant} (RPG-2019-034). KJF and PS were supported by a Royal Society International Exchange grant IES\textbackslash R1\textbackslash 191195 and PS by Project PICT 2015-3675 (ANPCyT).

\end{document}